\documentclass[12pt]{amsart}

\usepackage{ucs}

\usepackage{amssymb}
\usepackage{amsthm}
\usepackage{amsmath}
\usepackage{latexsym}
\usepackage[cp1251]{inputenc}
\usepackage{graphicx}
\usepackage{wrapfig}
\usepackage{caption}
\usepackage{subcaption}
\usepackage{indentfirst}
\usepackage[left=2.5cm,right=2.5cm,top=2.5cm,bottom=2.5cm,bindingoffset=0cm]{geometry}
\usepackage{enumerate}

\def\cB{{\mathcal B}}

\def\cK{{\mathcal K}}
\def\cL{{\mathcal L}}

\DeclareMathOperator{\alt}{Alt}
\DeclareMathOperator{\aut}{Aut}
\DeclareMathOperator{\AGL}{AGL}

\DeclareMathOperator{\cay}{Cay}

\DeclareMathOperator{\Inn}{Inn}

\DeclareMathOperator{\orb}{Orb}
\DeclareMathOperator{\out}{Out}

\DeclareMathOperator{\soc}{Soc}

\DeclareMathOperator{\sym}{Sym}

\DeclareMathOperator{\poly}{poly}

\DeclareMathOperator{\Reg}{Reg}

\DeclareMathOperator{\hol}{Hol}

\makeatletter
\def\@seccntformat#1{\csname the#1\endcsname. }
\def\@biblabel#1{#1.}

\makeatother

\title[On Cayley representations of central Cayley graphs]{On Cayley representations of central Cayley graphs over almost simple groups}

\author{Jin Guo, Wenbin Guo}
\address{School of Science, Hainan University, Haikou, Hainan, 570228, P.R. China}
\email{guojinecho@163.com, \ wbguo@ustc.edu.cn}

\author{Grigory Ryabov}
\address{Sobolev Institute of Mathematics, Novosibirsk, Russia}
\address{Novosibirsk State Technical University, Novosibirsk, Russia}
\email{gric2ryabov@gmail.com}

\author{Andrey V. Vasil'ev}
\address{Sobolev Institute of Mathematics, Novosibirsk, Russia}
\address{School of Science, Hainan University, Haikou, Hainan, 570228, P.R. China}
\email{vasand@math.nsc.ru}

\thanks{J. Guo is supported by the National Natural Science Foundation of China (No. 11961017),  W. Guo and A. V. Vasil'ev are supported by National Natural Science Foundation of China (No. 12171126), G. Ryabov is supported by the Mathematical Center in Akademgorodok under agreement No.~075-15-2022-281 with the Ministry of Science and Higher Education of the Russian Federation.}

\date{}

\newtheorem{prop}{Proposition}[section]
\newtheorem{theorem}{Theorem}

\newtheorem{lemm}[prop]{Lemma}
\newtheorem{theo}[prop]{Theorem}
\newtheorem{corl}[prop]{Corollary}

\theoremstyle{definition}
\newtheorem{defn}[theorem]{Definition}
\newtheorem*{rem}{Remark}

\begin{document}

%
%

\begin{abstract}
A Cayley graph over a group $G$ is said to be \emph{central} if its connection set is a normal subset of $G$. We prove that every central Cayley graph over a simple group $G$ has at most two pairwise nonequivalent Cayley representations over $G$ associated with the subgroups of $\sym(G)$ induced by left and right multiplications of~$G$. We also provide an algorithm which, given a central Cayley graph $\Gamma$ over an almost simple group $G$ whose socle is of a bounded index, finds the full set of pairwise nonequivalent Cayley representations of $\Gamma$ over $G$ in time polynomial in size of~$G$.
\end{abstract}

\maketitle

\section{Introduction}

Let $G$ be a group (throughout the paper all groups are assumed to be finite) and $X$ a subset of $G$ not containing the identity element. By a \emph{Cayley graph} $\cay(G,X)$ over $G$ with \emph{connection set} $X$, we mean a directed graph  with vertex set $G$ and arc set $\{(g,xg):~g\in G,~x\in X\}$. A \emph{Cayley representation} of an arbitrary directed graph $\Gamma$ over a group $G$ is defined to be a Cayley graph over~$G$ isomorphic to~$\Gamma$. Two Cayley representations of $\Gamma$ over $G$ are \emph{equivalent} if the corresponding Cayley graphs are \emph{Cayley isomorphic}, i.e., there exists a group automorphism of $G$ which is at the same time an isomorphism between the graphs. In general, a given graph may have many Cayley representations over a given group even up to equivalence, see, e.g.,~\cite{Li}. The problem of finding pairwise nonequivalent Cayley representations of a given graph is of substantial interest from a computational complexity point of view. For example, the isomorphism problem for Cayley graphs can be polynomially reduced to it for groups having a bounded number of generators. In the present paper, we are interested in a special case of this problem formed by the two conditions imposed on the input graphs and groups.


First, we assume that a Cayley graph $\Gamma=\cay(G,X)$ is {\em central}, i.e., the connection set $X$ of $\Gamma$ is normal: $X^g=X$ for every $g\in G$. Obviously, every Cayley graph over an abelian group is central, it explains why central graphs are also called quasiabelian in \cite{WX,Zgr}. The problem of finding pairwise nonequivalent Cayley representations was solved for cyclic groups~\cite{EP} and for some abelian groups of small rank~\cite{NP,Ry}. We will show that such results are still possible even for groups that are very far from abelian, if we restrict ourselves to {\em central Cayley representations} that is Cayley representations of central Cayley graphs.

Namely, and this is our second condition, we suppose that $G$ is a simple or almost simple group. Note that Cayley graphs over simple groups were a subject of intense interest for many years. It suffices to mention the positive solution of the Babai--Kantor--Lubotzky conjecture \cite{BKL} stating that all nonabelian simple groups are expanders in a uniform way, see, e.g., \cite[Conjecture~2.5]{Lub}. Some general description of the full automorphism group of a connected Cayley graph over a simple group was given in~\cite{FPW}. A polynomial algorithm solving the isomorphism problem for central Cayley graphs over almost simple groups was proposed in~\cite{PV}.

Our first result concerns simple groups. It is worth noting that it is based (see details in Section~\ref{sec:simple}) on the following fact: every nonabelian simple group $G$ contains a class of involutions invariant with respect to $\aut(G)$ \cite[Lemma~12.1]{FGS}.

\begin{theorem}\label{th:crsg}
Every directed graph $\Gamma$ has at most two pairwise nonequivalent central Cayley representations over a simple group $G$. Moreover, if $\Gamma$ is undirected, then all its central Cayley representation over $G$ are equivalent.
\end{theorem}

\begin{rem} The second statement of the theorem has an equivalent reformulation: every undirected central Cayley graph over a simple group is a {\em CI-graph}, cf. \cite[Definition~3.1]{Li}.
\end{rem}

As we will see in Section~\ref{sec:example}, the number of pairwise nonequivalent central Cayley representations of a graph over an almost simple group $G$ cannot be bounded by any constant not depending on the size of~$G$. For example, the {\em complete transposition graph} over the symmetric group $\sym(m)$ with connection set consisting of all transpositions of $\sym(m)$ has at least $[\frac{m}{4}]$ such representations (see Corollary~\ref{cor:example} and the remark after it).

Nevertheless, we show that the number of pairwise nonequivalent central Cayley representations of a graph $\Gamma$ over an almost simple group $G$ of order $n$ with the socle $\soc(G)$ of index at most~$c$ is bounded by $n^{f(c)}$, where $f$ is a function not depending on $n$ (Theorem~\ref{th:baseAS}). Moreover, the following holds.

\begin{theorem}\label{th:algorithm}
Let $\cK_c$ be the class of almost simple groups $G$ with $|G:\soc(G)|\leq c$. Given a central Cayley graph $\Gamma$ over $G\in\cK_c$ of order $n$, the full set of pairwise nonequivalent Cayley representations of~$\Gamma$ can be found in time polynomial in~$n$.
\end{theorem}

\begin{rem} We do not know whether the number of pairwise nonequivalent Cayley representations of~$\Gamma$ is bounded by a polynomial in~$n$ of degree not depending on the index~$c$. If it does, then this would be interesting to prove that such representations can be found within the same time.
\end{rem}

Slightly generalizing Babai's argument from~\cite{Babai}, one can prove (see Lemma~\ref{lem:Babai} below) that the pairwise nonequivalent Cayley representations of a Cayley graph $\Gamma$ over any group $G$ are in one-to-one correspondence with the conjugacy classes of regular subgroups of the group $\aut(\Gamma)$ isomorphic to~$G$. Thus, the following definition is a key to our arguments.

\begin{defn}\label{def:base}
A $G$-base $\cB$ of a permutation group $K$ is a maximal set of pairwise non-conjugate (in $K$)
regular subgroups of $K$ isomorphic to~$G$. If $\Gamma$ is a graph and $K=\aut(\Gamma)$, then $\cB$ is called a $G$-base of $\Gamma$.
\end{defn}

\begin{rem} According to our approach, in Definition~\ref{def:base} and further, by a graph we mean a directed graph.
\end{rem}

The notion of a $G$-base was suggested in~\cite{EMP} as a generalization of the notion of a \emph{cycle base} (see~\cite{EP,Muz}). One can check that all $G$-bases of $K$ have the same size denoted by $b_G(K)$ (and $b_G(\Gamma)$ if $K=\aut(\Gamma)$).

In this language (cf. Corollary~\ref{cor:base} below), Theorem~\ref{th:crsg} states that $b_G(\Gamma)\leq2$ if $\Gamma$ is a central Cayley graph over a simple group $G$. In fact, we prove that each regular subgroup of $\aut(\Gamma)$ is conjugate to either the group of right multiplications by elements of $G$ or the group of left multiplications. While in the case of an almost simple group $G$ with socle of index at most~$c$, the number $b_G(\Gamma)$ is bounded by $n^{f(c)}$, and Theorem~\ref{th:algorithm} claims that a $G$-base can be found within the same time.

\section{Regular subgroups and Cayley representations}\label{sec:crrs}

Let $\Omega$ be a finite set of size $n$ and let $K\leq\sym(\Omega)$ be a permutation group on $\Omega$. Recall that $K$ is called {\em semiregular} if each point stabilizer of $K$ is trivial; and $K$ is {\em regular} if $K$ is transitive and semiregular. Clearly, $K$ is semiregular if and only if $|K|=|\Delta|$ for every $\Delta$ from the set $\orb(K,\Omega)$ of orbits of $K$ on~$\Omega$. In the proof of Theorem~\ref{th:algorithm}, we need the following elementary fact.

\begin{lemm}\label{semireg}
Let $K$ and $M$ be semiregular subgroups of $\sym(\Omega)$, and $f$ an isomorphism from $K$ to~$M$. Then there exists $x\in\sym(\Omega)$ such that $g^f=g^x$ for every $g\in K$.
\end{lemm}

\begin{proof}
Let $\orb(K,\Omega)=\{\Delta_1,\ldots,\Delta_m\}$ and $\orb(M,\Omega)=\{\Lambda_1,\ldots,\Lambda_m\}$, fix elements $\alpha_i\in \Delta_i$ and $\beta_i\in \Lambda_i$ for every $i\in\{1,\ldots,m\}$. Define $x$ as follows:
$$(\alpha_i^g)^x=\beta_i^{g^{f}}$$
for every $i\in\{1,\ldots,m\}$ and every $g\in K$. Since $K$ and $M$ are semiregular, $x$ is a well-defined bijection from $\Omega$ to itself.

Now take any $g\in K$ and $\alpha\in \Omega$. There exist the unique $i\in\{1,\ldots,m\}$ and $t\in K$ such that $\alpha_i^t=\alpha^{x^{-1}}$. By definition of $x$,
$$\alpha^{x^{-1}gx}=(\alpha_i^{tg})^x=\beta_i^{(tg)^f}=((\beta_i)^{t^f})^{g^f}=((\alpha_i^t)^x)^{g^f}=((\alpha^{x^{-1}})^x)^{g^f}=\alpha^{g^f}.$$
Thus, $g^x=g^f$ for every $g\in K$.
\end{proof}

Given a group $G$ and its element $g$ (its subgroup $L$), denote by $g_l$ and $g_r$ ($L_l$ and $L_r$) the elements (subgroups) of $\sym(G)$ induced by left and right multiplications by $g$ (by elements of $L$) respectively. Observe that $L_l$ and $L_r$ are always permutable, so $L^*=L_lL_r$ is a subgroup of $\sym(G)$. In this notation, $G_l$ and $G_r$ are regular subgroups of $\sym(G)$ isomorphic to~$G$, the intersection $G_l\cap G_r$ is isomorphic to the center~$Z(G)$ of~$G$, and if $Z(G)=1$ then $G^*=G_l\times G_r$ is the direct product of $G_l$ and~$G_r$. Further we refer to the set of all regular subgroups of $K\leq \sym(\Omega)$ isomorphic to $G$ as $\Reg(K,G)$ (this set can be empty).
\medskip

Let $\Gamma$ be a graph with vertex set $\Omega$, and let $K=\aut(\Gamma)$. Then $\Gamma$ has a Cayley representation over a group $G$ if and only if $\Reg(K,G)\neq\varnothing$. Indeed, if $\Gamma=\cay(G,X)$, then $G_r\in\Reg(K,G)$. Vise versa, for every $H\in\Reg(K,G)$ and a
permutation group isomorphism $f:H\to G_r$, one can construct a Cayley representation of $\Gamma$ over $G$ as follows. Let
$f_0:\Omega\to G$ be the bijection induced by~$f$, i.e.,
$$
f_0(\alpha^h)=f_0(\alpha)^{f(h)}
$$
for all $\alpha\in\Omega$ and $h\in H$. Denote by $X_0$  the set of all neighbors of $\alpha_0=f_0^{-1}(e)$ in $\Gamma$,
where $e$ is the identity element of~$G$. Then one can see that $f_0$ is an isomorphism from $\Gamma$ onto the graph
$\cay(G,X)$, where $X=f_0(X_0)$. Clearly, this graph is a Cayley representation of $\Gamma$ over $G$; we say that this representation is {\it associated} with $H$ and $f$.  It is easy to see that any Cayley representation of $\Gamma$ over $G$ is
associated with a suitable regular group $H\le\aut(\Gamma)$  and an isomorphism $f:H\to G_r$.

\begin{lemm}\label{lem:Babai}
Let $\Gamma_i$ be a Cayley representation of a graph $\Gamma$, associated with a regular group $H_i\le\aut(\Gamma)$
and an isomorphism $f_i:H_i\to G_r$, $i=1,2$. Then $\Gamma_1$ and $\Gamma_2$ are equivalent if and only if
$H_1$ and $H_2$ are conjugate in $\aut(\Gamma)$.
\end{lemm}

\begin{proof} If $H_2=(H_1)^\gamma$ for some $\gamma\in\aut(\Gamma)$, then $\sigma=f_1^{-1}\gamma f_2$ is an isomorphism
from $\Gamma_1$ onto~$\Gamma_2$. Since $\sigma$ takes $G_r$ to itself preserving the identity element, it is a Cayley isomorphism. Thus,
the Cayley representations $\Gamma_1$ and $\Gamma_2$ are equivalent. Conversely, if $\sigma\in\aut(G)$ is Cayley isomorphism
from $\Gamma_1$ onto $\Gamma_2$, then $\gamma=f_1\sigma f_2^{-1}$ is an automorphism of $\Gamma$ such that $H_2=(H_1)^\gamma$.
\end{proof}

\begin{corl}\label{cor:base}
For every $G$-base $\cB$ of a graph $\Gamma$, the Cayley representations of $\Gamma$
associated with the groups $H\in\cB$, form the full system of pairwise nonequivalent Cayley representations
of $\Gamma$ over~$G$. In particular, the size of this system is equal to $b_G(\Gamma)$.
\end{corl}

According to our definition of a Cayley graph $\Gamma=\cay(G,X)$, the automorphism group $K=\aut(\Gamma)$ always contains the group $G_r$ of right multiplications. If $\Gamma$ is central, we can say more.

\begin{lemm}\label{lem:central}
Suppose that a Cayley graph $\Gamma=\cay(G,X)$ is central, and $K=\aut(\Gamma)$. Then the following hold:
\begin{enumerate}
\item[{\em(i)}] $G_l\leq K;$ in particular, $G^*=G_lG_r\leq K.$
\item[{\em(ii)}] $X=X^{-1}$ if and only if the permutation $\sigma:g\mapsto g^{-1}$ belongs to~$K;$ in particular, if $\Gamma$ is undirected, then $G_l$ and $G_r$ are conjugate in~$K$.

\end{enumerate}
\end{lemm}

\begin{proof} (i) For every $g\in G$ and every $(h,xh)\in E(\Gamma)$, where $E(\Gamma)$ is the arc set of~$\Gamma$,
$$
(h,xh)^{g_l}=(g^{-1}h,g^{-1}xh)=(g^{-1}h,x^gg^{-1}h)\in E(\Gamma),
$$
because $X$ is normal.

(ii) Obviously, if $\sigma\in K$, then $X=X^{-1}$, and $G_l=G_r^\sigma$ and $G_r$ are conjugate in~$K$. If $X=X^{-1}$, then $$(h,xh)^\sigma=(h^{-1},h^{-1}x^{-1})=(h^{-1},(x^{-1})^hh^{-1})\in E(\Gamma),$$ for every $(h,xh)\in E(\Gamma)$.
\end{proof}

\begin{rem} It is easy to see that there are different ways to define central Cayley graphs and describe their basic properties, see, e.g., \cite[Section~1.1 and Corollary~2.5]{Li2}, where such graphs were called {\em holomorph Cayley graphs}.
\end{rem}

\section{Central Cayley graphs over simple groups}\label{sec:simple}

It is clear that any two isomorphic regular subgroups of the symmetric group are conjugate in it, so  Lemma~\ref{lem:Babai} and Lemma~\ref{lem:central}(ii) yield that in order to prove Theorem~\ref{th:crsg}, it suffices to establish the following assertion.

\begin{theo}\label{main}
Suppose that $\Gamma$ is a central Cayley graph over a simple group $G$, and $K=\aut(\Gamma)$. Then either $K=\sym(G)$ or $\Reg(K,G)=\{G_{l}, G_{r}\};$ in particular, $b_G(\Gamma)\leq2$.
\end{theo}

\begin{proof}

We start with a simple observation that the only $2$-transitive group which is the full automorphism group of a graph with the vertex set $\Omega$ is $\sym(\Omega)$. So in the case when $G$ is abelian, we are done by the Burnside theorem ~\cite{Burn}. Indeed, if $K$ is not $2$-transitive, then $K\leq\AGL(1,p)$ for a prime $p$, so $G_r$ being normal is the only subgroup of order $p$ in~$K$. Therefore, we may further assume that $G$ is nonabelian simple and $K$ is not $2$-transitive.\medskip

Following \cite{LPS}, for a group $G$ we define the holomorph $\hol(G)$ as the normalizer $N_{\sym(G)}(G_r)$ and set $D(2,G)$ to be the subgroup of $\sym(G)$ generated by  $\hol(G)$ and the permutation $\sigma:g\mapsto g^{-1}$. Since $G$ is nonabelian simple, groups $G_{l}$ and $G_{r}$ intersect trivially, hence
\begin{equation}\label{eq:G*}
G^{*}=G_{l}\times G_{r}=\Inn(G)\ltimes G_r\leq D(2,G)=\langle\sigma\rangle\ltimes(\aut(G)\ltimes G_r).
\end{equation}

It can be deduced form the classification of finite simple groups (CFSG), see, e.g, Tables~5,6 in p.~xvi \cite[Introduction]{CCNPW}, or~\cite{Kohl}, that
\begin{equation}\label{eq:Out}
|\out(G)|\leq\log n
\end{equation}
(hereinafter, all the logarithms are binary and $|G|=n$.). This fact and \eqref{eq:G*} directly yield

\begin{lemm}\label{l0}
$|D(2,G)|\leq 2n^2 \log n$.
\end{lemm}

The next lemma essentially proved in~\cite{PV} exploits the classification of regular almost simple subgroups of a primitive group~\cite[Theorem~1.4]{LPS}.

\begin{lemm}\label{l2}
$G^*\leq K\leq D(2,G)$.
\end{lemm}

\begin{proof}
The first inclusion follows from Lemma~\ref{lem:central}(i). If $K$ is imprimitive, then $G$ has a nontrivial proper normal subgroup by~\cite[Lemma~4.2]{PV} that in our case contradicts to simplicity of~$G$. So $K$ must be uniprimitive (primitive but not $2$-transitive). The rest follows from~\cite[Lemma~3.2]{PV}.
\end{proof}

The key lemma below is a reformulation of \cite[Lemma~12.1]{FGS}.

\begin{lemm}\label{l1}
If $G$ is a nonabelian simple group, then there exists an involution $t\in G$ such that $\aut(G)=C_{\aut(G)}(t)\Inn(G)$.
\end{lemm}

The next lemma completes the proof of the theorem.

\begin{lemm}\label{glgr}
$\Reg(K,G)=\{G_{l}, G_{r}\}$.
\end{lemm}

\begin{proof}
Let $H\in \Reg(K,G)$. The group $G^{*}$ is clearly normal in $D(2,G)$, so $H\cap G^{*}$ is normal in $H$. Since $H$ is simple, $H\cap G^{*}=H$ or $|H\cap G^{*}|=1$. In the latter case, $|HG^{*}|=n^3\leq |K|\leq |D(2,G)|$, a contradiction to Lemma~\ref{l0}. Thus, $H\leq G^{*}$.

Assume to the contrary that $H$ is neither $G_{l}$ nor $G_{r}$.  Then, due to simplicity of~$G$, the projections of $H\leq G_{l}\times G_{r}$ to each of the two factors are isomorphic to~$G$. It follows that there exists an automorphism $\tau\in\aut(G)$ such that
$$H=G_{\tau}=\{g_{l}(g^\tau)_{r}: g\in G\}.$$

Let us choose an involution $t\in G$ as in Lemma~\ref{l1}. Since $\tau\in\aut(G)$, there exist $\varphi\in C_{\aut(G)}(t)$ and $\psi\in\Inn(G)$ such that $\tau=\varphi\psi$. Fix any $x\in G$ such that $\psi$ is a conjugation by~$x$. Consider the permutation $h=t_{l}(t^\tau)_{r}$ from $H$. Clearly, $h$ is nontrivial.

Now, if $t\neq x$, then
$$
(tx)^h=(tx)^{t_{l}(t^\tau)_{r}}=t(tx)t^\tau=xt^{\varphi\psi}=
x (x^{-1}t^\varphi x)=tx.
$$
If $t=x$, then $t^\tau=(t^\varphi)^t=t^3=t$ and hence $t^h=t$. Thus in the both cases, the element $h\in G_{\tau}$ leaves at least one point fixed. Since $h$ is nontrivial, the group $H=G_\tau$ is not regular, a contradiction.
\end{proof}

\begin{rem} The proof of Lemma~\ref{glgr} is valid for any permutation group $K$ with $G^*\leq K\leq D(2,G)$.
\end{rem}\end{proof}

\section{Base of complete transposition graph}\label{sec:example}

A lower bound on the number of the pairwise nonequivalent central Cayley representation of the complete transposition graph mentioned in Introduction is a consequence of the following assertion.

\begin{theo}\label{th:example}
Suppose that $G=\sym(m)$ and $K=D(2,G)=\langle\sigma\rangle\ltimes(\aut(G)\ltimes G_r)$, where $\sigma:g\mapsto g^{-1}$. Then $b_G(K)\geq [\frac{m}{4}]$.
\end{theo}

\begin{proof} If $m\leq 7$, then $[\frac{m}{4}] \leq 1$ and the theorem is obvious. Further we assume that $m\geq 8$. This condition provides that $\aut(G)=\Inn(G)$. Therefore, $K=\langle\sigma\rangle\ltimes G^*$, where $G^*=G_l\times G_r;$ every $k\in K$ is uniquely represented as
\begin{equation}\label{eq:k}
k=\sigma^\varepsilon u_l z_r,\mbox{ where }\varepsilon\in\{0,1\},\mbox{ and }u, z\in G;
\end{equation}
and $z_r^\sigma=z_l$ for every $z\in G$.

Set $A=\soc(G)=\alt(m)$ and fix an element $x\in G\setminus A$. For every $y\in G$ set $\tau_y=y_lx_r\in G^*$ and $H_y=\langle\tau_y\rangle A_r$. Since $A_r$ is normal in $G^*$, $H_y$ is a subgroup of $G^*$, and $H_y=G_r$ if and only if $y=1$.

\begin{lemm}\label{tr1}
If $t$ is an involution from $A$, then $H_t$ is a regular subgroup of $K$, isomorphic to~$G$.
\end{lemm}

\begin{proof} Assume that $g=g^{\tau_tz_r}=tgxz$ for some $g\in G$ and $z\in A$. Since $t$ and $z$ are even permutations, while $x$ is an odd one, the permutations $g$ and $tgxz$ are of distinct parities, a contradiction. Therefore, $H_t$ is a semiregular subgroup. In fact, it is a regular subgroup due to $|H_t|=|G|$. Indeed, $H_t>A_r$, because $t_l\neq1$, and $|H_t:A_r|\leq2$, because $\tau_t^2\in A_r$.

Let $N$ be a nontrivial normal subgroup of $H_t$ distinct from $A_r$. Since $A_r$ is simple and $|H_t:A_r|=2$, it follows that $H_t=N\times A_r$, where $|N|=2$. Let $n=t_lx_rz_r$ be the involution in~$N$. It follows that $u_rn=nu_r$ and, consequently, $uxz=xzu$ for every $u\in A$. Then $xz\in C_G(A)$, where $C_G(A)$ is trivial due to simplicity of $A$. Thus, $xz=1$, which is impossible, because $x$ and $z$ are the permutations of distinct parities. Thus, $A_r=\soc(H_t)$, hence $A_r\unlhd H_t\leq\aut(A_r)\simeq G$. Finally, $|H_t|=|G|$ implies $H_t\simeq G$.
\end{proof}

\begin{lemm}\label{tr2}
Let $t$ be an involution from $A$ and $y\in G$. If $y\neq t$, then $y_lz_r\not\in H_t$ for every $z\in G$. In particular, $H_y=H_t$ if and only if $y=t$.
\end{lemm}

\begin{proof} Assume that $y_lz_r\in H_t$ for some $z\in G$. Then $y_lz_{r}=t_lx_ru_{r}$ for some $u\in A$ due to Lemma~\ref{tr1}. Since $G^*=G_l\times G_r$, we conclude that $y=t$, a contradiction.
\end{proof}

\begin{lemm}\label{tr3}
Let $t$ and $y$ be involutions from $A$. If $H_t$ and $H_y$ are conjugate in $K$, then $t$ and $y$ are conjugate in~$G$.
\end{lemm}

\begin{proof} Suppose that $H_t$ and $H_y$ are conjugate by $k=\sigma^\varepsilon u_l z_r\in K$, cf. \eqref{eq:k}. Straightforward computations yield
$$(\tau_t)^k=\begin{cases}
(utu^{-1})_l(z^{-1}xz)_r\mbox{ for }\varepsilon=0,\\
(ux^{-1}u^{-1})_l(z^{-1}t^{-1}z)_r\mbox{ for }\varepsilon=1.
\end{cases}$$
Since $(\tau_t)^k\in H_y$, Lemma~\ref{tr2} implies that either $y=t^{u^{-1}}$ or $y=(x^{-1})^{u^{-1}}$. The latter case is impossible, because $y$ is an even permutation, while $x^{-1}$ is odd.
\end{proof}

Let $T$ be the set of involutions from $\soc(G)=\alt(m)$ pairwise non-conjugate in $G$. Lemmas~\ref{tr1} and~\ref{tr3} yield $b_G(K)\geq |T|\geq [\frac{m}{4}].$
\end{proof}

\begin{corl}\label{cor:example}
Let $G=\sym(m)$, $m\geq 5$, $X=X^{-1}$ a proper normal subset of $G\setminus\soc(G)$, and $\Gamma=\cay(G,X)$. Then the number of pairwise nonequivalent central Cayley representations of $\Gamma$ is at least $[\frac{m}{4}]$.
\end{corl}

\begin{proof} By~\cite[Theorem~1.3]{PV}, $\aut(\Gamma)=D(2,G)$. The rest follows from Corollary~\ref{cor:base} and Theorem~\ref{th:example}.
\end{proof}

\begin{rem} The complete transposition graph over $G=\sym(m)$ mentioned in Introduction is a particular case of graphs from the corollary (if $m\geq5$).  If $m<5$, then the conclusion of the corollary remains obviously true for such a graph. It is worth noting that the equality $\aut(\Gamma)=D(2,G)$ for the complete transportation graph was obtained in~\cite[Theorem~1.1]{Gan}.
\end{rem}

\section{Central Cayley graph over almost simple groups}\label{sec:algorithm}

In order to present our arguments we need some notation concerning imprimitive permutation groups. First, if $\Delta\subseteq\Omega$ is a block of $K\leq\sym(\Omega)$, then the group $\{f^{\Delta}:~f\in K, \Delta^f=\Delta\}$ is denoted by $K^{\Delta}$, where $f^\Delta$ is the induced bijection from $\Delta$ to~$\Delta^f$.
For an imprimitivity system $\mathcal{L}$ of a transitive group $K$, set $K_{\mathcal{L}}$ and $K^{\mathcal{L}}$ to be, respectively, the intersection of the setwise stabilizers $K_{\{\Delta\}}$ for all $\Delta\in\mathcal{L}$ and the permutation group induced by the action of~$K$ on $\mathcal{L}$.

Let $G$ be a finite group and $K\leq \sym(G)$ such that $K\geq G^{*}=G_lG_r$. If $X$ is a block of $K$ containing the identity element~$e$ of~$G$, then $X$ is a subgroup of $G$ because $K\geq G_r$ (cf.~\cite[Theorem~24.12]{Wi}); moreover, $X$ is normal in $G$ because $K\geq G_{l}$. Denote by~$L=L(K)$ the intersection of all non-singleton blocks of $K$ containing~$e$. Following~\cite{PV}, we say that $L$ is the \emph{minimal block} of $K$. Denote the block system of $K$ containing $L$ by $\mathcal{L}=\mathcal{L}(K)$. Clearly, $\mathcal{L}$ coincides with the set of all $L$-cosets, $L^{*}\leq K_{\mathcal{L}}$, and $\orb(L^{*},G)=\mathcal{L}$.\medskip

The following assertion is our key argument in proving Theorem~\ref{th:algorithm}. Recall that $\cK_c$ is the class of almost simple groups $G$ with $|G:\soc(G)|\leq c$.

\begin{theo}\label{th:baseAS}
If $\Gamma$ is a central Cayley graph over $G\in\cK_c$, then $b_G(\Gamma)$ is polynomial in~$n$.
\end{theo}

\begin{proof}
Let $K=\aut(\Gamma)$. It suffices to point out a set of subgroups of~$K$, which includes some $G$-base of $K$ and has the size polynomial in~$n$.
The lemma below summarizes some properties of $K$ proved in~\cite[Lemma~4.2, Proposition~6.1, and Theorem~6.2]{PV}.

\begin{lemm}\label{block}
Let $L=L(K)$ be the minimal block of $K$, and $\mathcal{L}=\mathcal{L}(K)$. Then $L$ is an almost simple group containing the socle $S$ of $G$ and one of the following holds:
\begin{enumerate}

\item[{\em(i)}]  $L=S$ and $K^L\leq D(2,L);$

\item[{\em(ii)}] $K=(K_{\mathcal{L}})\wr K^{\mathcal{L}}$ is the full wreath product and $(K_{\mathcal{L}})^\Delta=\sym(\Delta)$ for every $\Delta\in \mathcal{L}$.
\end{enumerate}
\end{lemm}

Following~\cite{PV}, we say that $K$ is of {\em normal type} provided Item~(i) of Lemma~\ref{block} holds, and of {\em symmetric type} otherwise.

Let $K$ be of normal type. Since $K\leq K^L\wr K^{\mathcal{L}}$ and $|\cL|\leq c$, Item~(i) of Lemma~\ref{block} and Lemma~\ref{l0} yield that $$|K|\leq|K^L|^c|\sym(\cL)|\leq(2n^2\log n)^c\cdot c!=\poly(n).$$  By~\cite{DL}, every almost simple group is $3$-generated. Hence the number of the subgroups of $K$ isomorphic to $G$ is polynomial in the size of~$K$. It follows that $|\Reg(K,G)|$ is polynomial in~$n$, and so is any $G$-base of~$K$.

Therefore, we may suppose that $K$ is of symmetric type. By Lemma~\ref{block}, the minimal block $L$ is an almost simple group with the socle $S$, hence every $H\in\Reg(K,G)$ includes a normal subgroup isomorphic to~$L$. If $T=\soc(H)$ does not stabilize each $\Delta\in\cL$, then $T$ acts faithfully on the set $G/L$ of size at most~$c$. In this case, the order of $G$ is bounded by a function depending only on~$c$, and we are obviously done.  Therefore, $T$ is a semiregular subgroup of~$K_{\mathcal{L}}$, in particular, the same is true for $S_r=\soc(G_r)$. Lemma~\ref{semireg} implies that $T^\Delta$ and $S_r^\Delta$ are conjugate in $(K_{\mathcal{L}})^\Delta=\sym(\Delta)$ for each $\Delta\in\cL$. However, $K_{\mathcal{L}}$ is the direct product of the symmetric groups $\sym(\Delta), \Delta\in\cL$, by Lemma~\ref{block}(ii), so there is $k\in K_{\mathcal{L}}$ such that $T=S_r^k$.

Thus, $T$ is conjugate to $S_r$ in~$K$. It follows that $H$ is conjugate (in $K$) to a regular subgroup from the normalizer $N=N_{K}(S_r)=N_{\sym(G)}(S_r)\cap K$ of $S_r$ in~$K$. It is easy to see that $|N|=\poly(n)$. Indeed, $|N_{\sym(G)}(S_r)|\leq|\aut(S_r)||C|$, where $C=C_{\sym(G)}(S_r)$ is the centralizer of~$S_r$. Since $S_r$ is a simple group, $|\aut(S_r)|\leq n\log n$, see~\eqref{eq:Out}. On the other hand, the set $\orb(S_r,G)$ is an imprimitivity system for $C$, so $C$ is embedded in the corresponding imprimitive wreath product. It follows that $|C|\leq|C_{\sym(\Delta)}(S_r)|^c|\sym(c)|$ for $\Delta\in\orb(S_r,G)$.  However, by semiregularity of $S_r$, $|C_{\sym(\Delta)}(S_r)|=|S_r|=n$, and we are done. Thus, if $K$ is of symmetric type, then the size of $\Reg(N,G)$ is polynomial in~$n$, and so is the size of a $G$-base in $K$.  \end{proof}

\textbf{Proof of Theorem~\ref{th:algorithm}.} Let $\Gamma$ be a central Cayley graph over $G\in\cK_c$. According to Corollary~\ref{cor:base}, in order to find the full system of pairwise nonequivalent central Cayley representations of~$\Gamma$, it suffices to find a $G$-base $\cB$ of $\Gamma$, that is a maximal set of pairwise non-conjugate regular subgroups of $K=\aut(\Gamma)$, isomorphic to $G$. Theorem~\ref{th:baseAS} implies that the size of $\cB$ is polynomial in~$n$. Here we show that $\cB$ can be found efficiently.

We begin with the observation that a generating set of $K$ can be found in time $\poly(n)$ by~\cite[Corollary~1.2]{PV}. Applying standard algorithms for permutation groups (here our main source is~\cite{Seress}), one can find the socle $S_r=\soc(G_r)$, the minimal block $L=L(K)$ and $\cL=\cL(K)$ in time $\poly(n)$, and check whether $K$ is of symmetric type or normal type.

If $K$ is of symmetric type, then put $M=N_K(S_r)=N_{\sym(G)}(S_r)\cap K$. Note that $N_{\sym(G)}(S_r)$ can be found in time $\poly(n)$ due to~\cite[Corollary~3.24]{LM}. If $K$ is of normal type, we simply put $M=K$. In both cases, arguments from the proof of Theorem~\ref{th:baseAS} show that $|M|=\poly(n)$ and the set $\Reg(M,G)$ contains a $G$-base of~$K$.

The set $\Reg(M,G)$ can be found in time polynomial in~$|M|$ by standard procedures. Indeed, it suffices to iterate all $3$-generated subgroups of $M$ and for each of them check whether it is regular and isomorphic to~$G$. In particular, $|\Reg(M,G)|=\poly(n)$.

It remains to choose subgroups in $\Reg(M,G)$ pairwise non-conjugate in~$K$. Subgroups $H_1, H_2\in\Reg(K,G)$ are conjugate in $K$, if and only if the corresponding Cayley graphs $\Gamma_1$ and $\Gamma_2$ are Cayley isomorphic (Lemma~\ref{lem:Babai}). However, testing Cayley isomorphism of two Cayley graphs over $G$ can be easily done in time polynomial in the order of $\aut(G)$, which, in our case, is polynomial in~$n$, because $G$ is almost simple (see~\eqref{eq:Out}). This completes the proof of Theorem~\ref{th:algorithm}.\medskip

\noindent\textbf{Acknowledgments.} The authors are grateful to Prof. Ilia Ponomarenko for fruitful discussions and helpful suggestions.\medskip

\noindent\textbf{Data availability.} Data sharing is not applicable to this article as no datasets were generated or analyzed.\medskip

\noindent{\bf Statements and Declarations.} The authors declare that they have no conflict of interest.


\begin{thebibliography}{list}


\bibitem{Babai}
\emph{L.~Babai}, Isomorphism problem for a class of point symmetric structures, Acta Math. Acad. Sci. Hung., \textbf{29}, No.~3 (1977), 329--336.

\bibitem{BKL}
\emph{L.~Babai, W.~M.~Kantor, A.~Lubotsky}, Small-diameter Cayley graphs for finite simple groups, European J. Combin., \textbf{10}, No.~6 (1989), 507--522.

\bibitem{Burn}
\emph{W.~Burnside}, Theory of groups of finite order, 2nd ed., Cambridge Univ. Press, London (1911).

\bibitem{CCNPW}
\emph{J.~H.~Conway, R.~T.~Curtis, S.~P.~Norton, R.~A.~Parker, R.~A.~Wilson}, Atlas of finite groups, Clarendon Press, Oxford (1985).

\bibitem{DL}
\emph{F.~Dalla Volta, A.~Lucchini}, Generation of almost simple groups, J. Algebra, \textbf{178}, No.~1 (1995), 194--223.

\bibitem{EP}
\emph{S.~Evdokimov,~I.~Ponomarenko}, Circulant grpahs: recognizing and isomorphism testing in polynomial time, St. Petersburg Math. J., \textbf{15}, No.~6 (2004), 813--835.

\bibitem{EMP}
\emph{S.~Evdokimov, M.~Muzychuk, I.~Ponomarenko}, A family of permutation groups with exponentially many non-conjugated regular elementary abelian subgroups, St. Petersburg Math. J., \textbf{29}, No.~4 (2018), 575--580.

\bibitem{FPW}
\emph{X.~G.~Fang, C.~E.~Praeger, J.~Wang}, On the automorphism groups of {C}ayley graphs of finite simple groups, J. London Math. Soc. (2), \textbf{66}, No.~3 (2002), 563--578.

\bibitem{FGS}
\emph{M.~D.~Fried, R.~Guralnick,~J.~Saxl}, Schur covers and Carlitz's conjecture, Isr. J. Math., \textbf{82} (1993), 157--225.

\bibitem{Gan}
\emph{A.~Ganesan}, Automorphism group of the complete transposition graph, J. Algebr. Comb., \textbf{42}, No.~3 (2015), 793--801.

\bibitem{Kohl}
\emph{S.~Kohl}, A bound on the order of the outer automorphism group of a finite simple group of given order, preprint (2003), 2 pp, https://stefan-kohl.github.io/preprints/outbound.pdf

\bibitem{Li}
\emph{C.~H.~Li}, On isomorphisms of finite Cayley graphs -- a survey, Discrete Math., \textbf{256}, Nos.~1-2 (2002), 301--334.

\bibitem{Li2}
\emph{C.~H.~Li}, Finite edge-transitive Cayley graphs and rotary Cayley maps, Trans. Amer. Math. Soc., \textbf{358}, No.~10 (2006), 4605--4635.


\bibitem{LPS}
\emph{M.~W.~Liebeck, Ch.~E.~Praeger, and J.~Saxl}, Regular subgroups of primitive permutation groups, Memoirs Amer. Math. Soc., {\bf 203}, No. 952 (2010), 1--88.

\bibitem{Lub}
\emph{A.~Lubotzky}, Expander graphs in pure and applied mathematics, Bull. Amer. Math. Soc. (N.S.), \textbf{49}, No.~1 (2012), 113--162.

\bibitem{LM}
\emph{E.~M.~Luks, T.~ Miyazaki}, Polynomial-time normalizers, Discrete Math. Theor. Comput. Sci., \textbf{13}, No.~4 (2011), 61--96.

\bibitem{Muz}
\emph{M.~Muzychuk}, On the isomorphism problem for cyclic combinatorial objects, Discr. Math., \textbf{197/198}, (1999), 589--606.

\bibitem{NP}
\emph{R.~Nedela, I.~Ponomarenko},  Recognizing and testing isomorphism of Cayley graphs over an abelian group of order~$4p$ in polynomial time, G. Jones et al (eds). Isomorphisms, Symmetry and Computations in Algebraic Graph Theory, Springer (2020), 195--218.

\bibitem{PV}
\emph{I.~Ponomarenko, A.~Vasil'ev}, Testing isomorphism of central Cayley graphs over almost simple groups in polynomial time, J. Math. Sci., \textbf{234}, No.~2 (2018), 219--236.

\bibitem{Ry}
\emph{G.~Ryabov}, On Cayley representations of finite graphs over abelian $p$-groups, St. Petersburg Math. J., \textbf{32}, No.~1 (2021), 71--89.

\bibitem{Seress}
\emph{A. Seress}, Permutation group algorithms, Cambridge Tracts Math., Cambridge University Press, Cambridge (2003).


\bibitem{WX}
\emph{J. Wang, M.-Y. Xu}, A class of hamiltonian Cayley graphs and Parsons graphs, European J. Combin., \textbf{18} (1997), 597--600.

\bibitem{Wi}
\emph{H.~Wielandt}, Finite permutation groups, Academic Press, New York - London (1964).



\bibitem{Zgr}
\emph{B.~Zgrablic}, On quasiabelian Cayley graphs and graphical doubly regular representations, Discr. Math., \textbf{244} (2002), 495--519.

\end{thebibliography}
\end{document}